\documentclass[12pt]{article}
\usepackage{enumerate}
\usepackage{amsmath}
\usepackage{amsthm}
\usepackage{amssymb}
\usepackage{algorithm}
\usepackage[bookmarksnumbered,  plainpages]{hyperref}

\newtheorem{thm}{Theorem}[section]
\newtheorem{ex}{Example}[section]

\newtheorem{lem}{Lemma}[section]
\newtheorem{mthd}{Method}[section]

\theoremstyle{definition}

\theoremstyle{remark}

\begin{document}
	
	\title{{New Accelerated  Modulus-Based  Iteration Method  for Solving Large and Sparse Linear Complementarity Problem}}
	
	\author{Bharat Kumar$^{a,1}$, Deepmala$^{a,2}$ and A.K. Das $^{b,3}$\\
		\emph{\small $^{a}$Mathematics Discipline,}\\
		\emph{\small PDPM-Indian Institute of Information Technology, Design and Manufacturing,}\\
		\emph{\small Jabalpur - 482005 (MP), India}\\
		\emph{\small $^{b}$Indian Statistical Institute, 203 B.T. Road, }\\
		\emph{\small Kolkata - 700108, India}\\
		\emph{\small $^1$Email:bharatnishad.kanpu@gmail.com , $^2$Email: dmrai23@gmail.com}\\
		\emph{\small $^3$Email: akdas@isical.ac.in}}
	\date{}
	\maketitle
	
	\abstract{\noindent 	In this article, we establish a class of new accelerated modulus-based iteration methods for solving the linear complementarity problem. When the system matrix is an $H_+$-matrix, we present appropriate criteria for the convergence analysis. Also, we demonstrate the effectiveness of our proposed method and reduce the number of iterations and CPU time to accelerate the convergence performance by providing two numerical examples for various parameters.}

	\noindent \textbf{Keywords.} Linear complementarity problem, Iteration  method, $P$-matrix, $H_{+}$-matrix, Convergence analysis, Matrix splitting.\\
	
	\noindent \textbf{Mathematics Subject Classification.} 90C33, 65F10, 65F50.\\
	
	
	\maketitle
	
	\section{Introduction}
	The large and sparse matrices are matrices that have a large number of rows and columns but a small number of non-zero elements. In other words, they are matrices where the majority of the elements are zero. Sparse matrices are commonly used to represent complex systems or large datasets in fields such as computer science, mathematics, physics and engineering. The sparsity of the matrix means that it is not practical to store each element individually and specialized data structures and algorithms must be used to efficiently store and manipulate the matrix.\\
	Given $A_1\in R^{n\times n}$ and  a vector  $\,q\,\in\,R^{n}.$ The linear complementarity problem, represented as LCP$(q, A_1)$, is to find the solution $z \in R^n$ to the following system:
	\begin{eqnarray}\label{eq1}
		z\geq 0, ~~~~    A_{1}z +q \geq 0,~~~~ z^T(A_{1}z +q)=0.	
	\end{eqnarray}
	The free boundary problem, the Nash equilibrium point of the bimatrix game, operations research, control theory, mathematical economics, optimization theory, stochastic optimal control, the American option pricing problem, and elasticity theory are among the applications of the linear complementarity problem that are extensively studied in the literature on mathematical programming. for more details see \cite{jana2021more}, \cite{dutta2022column}, \cite{neogy2011singular}, \cite{das2018some}, \cite{jana2018processability}, \cite{das2023more} and \cite{kumar2022error}.    \\
	The methods available for solving the linear complementarity problems are into two
	groups namely  the pivotal method \cite{Das2014}, \cite{Das2017} and the iterative method
	\cite{Najafi2013}, \cite{Hadjidimos2018}, \cite{Das2016}, \cite{Jana2019}, \cite{Kumar2022}, \cite{Kumar2023i} and  \cite{Kumar2023}. The basic idea behind the pivotal method is
	to get a basic feasible complementary vector through a series of pivot steps, while the
	iterative method creates a series of iterates that lead to a solution .
	Reformulating the LCP$(q, A_1)$ as an equation whose solution must be the same as the LCP $(q, A_1)$ is one of the most well-known and highly sought-after techniques for creating fast and economical iteration methods. As a result, some useful LCP$(q, A_1)$ equivalent forms have emerged. Mangasarian \cite{Mangasarian1977}  presented three methods: projected Jacobi over-relaxation, projected SOR, and projected symmetric SOR. For more information on designing iteration methods using the idea of Mangasarian, see also \cite{Ahn1981}, \cite{Bai1997} and  \cite{Yuan2003}. Bai in \cite{Bai2010} given the following general equivalent form:
	\begin{equation}\label{eq333}	
		(\Omega_{1} + M_{1})s = N_{1} s + (\Omega_{1} - A_{1})| s| - r q, 
	\end{equation} 
	with $r > 0$, where $ \Omega_{1} \in R^{n\times n}$ is a positive diagonal matrix, and initially, a class of modulus-based matrix splitting iteration algorithms was developed. The Equation (\ref{eq333}) covers the published works in \cite{Berman1979}, \cite{Van1980}, \cite{Kappel1986}, \cite{Dong2009} and  \cite{Hadjidimos2009}. This type of modulus-based matrix splitting iteration method was considered an effective method for solving the LCP$(q, A_1)$. For other deformations of  Equation (\ref{eq333}), see  \cite{Zheng2013}, \cite{Zhang2011},  \cite{Li2013}, \cite{Xu2015} and  \cite{Zhang2013} for more details. Moreover, this concept has also been used successfully in other complementarity problems, such as the nonlinear complementarity problem \cite{Ma2016}, \cite{Xia2015}, the implicit complementarity problem \cite{Huang2016}, \cite{Hong2016}, the quasi-complementarity problem \cite{Wu2018}, and the horizontal linear complementarity problem \cite{Mezzadri2020}.\\
Using the ideas of Shilang \cite{Shilang2021} and Bai \cite{Bai2010}, we present a class of new accelerated modulus-based iteration methods for solving the large and sparse LCP$(q, \mathcal{A})$. Also,  we show that the linear complementarity problem and fixed point equation are equivalent and provide some convergence domains for our proposed method. \\ 
The following is the structure of the article:  Some required definitions, notations and well-known lemmas are provided in section 2, all of which will be used for the discussions in the remaining sections of this work. In section 3, a new accelerated modulus-based iteration method with the help of the new equivalent fixed point form of the LCP$(q, \mathcal{A})$ is provided. In section 4, we establish some convergence domains for the proposed method. A numerical comparison between the proposed methods and the modulus-based matrix splitting methods, introduced by Bai \cite {Bai2010}, is illustrated in section 5. Section 6 contains the conclusion of the article.
	\section{Preliminaries}\label{Preli}
	In this part, we briefly discuss the basic results, definitions, and notations, most of which may be found in \cite{Xi2021}, \cite{Frommer1992} and  \cite{Neogy2012}.\\
	Suppose  $ A_{1}=( \bar{a}_{ij}) \in  {R}^{n\times n}$  and $ B_{1}=( \bar{b}_{ij}) \in { R}^{n\times n}$ are square matrices. The matrices $A_{1}$ and $B_{1}$ are denoted by  $A_{1} \geq $ $(\textgreater)$ $ B_{1}$ if $\bar{a}_{ij}\geq (\textgreater)$ $ \bar{b}_{ij}$ for all $ i,j \in \{1,2,\ldots,n\}$. The matrix  $| A_{1}|=(\bar{c}_{ij})$  is defined by $ \bar{c}_{ij}  = | \bar{a}_{ij}|$ $\forall ~i,j$ and $|A_{1}| $ represent that $ \bar{a}_{ij} \geq 0$ $\forall ~i,j $. Also, we have $|A_{1}+B_{1}|\leq |A_{1}|+|B_{1}|$  and $|A_{1}B_{1}|\leq |A_{1}||B_{1}|$. Moreover,  $a_{1}, b_{1} \in {R}^{n}$ then $|a_{1}+b_{1}|\leq |a_{1}|+|b_{1}|$ and  $||a_{1}|-|b_{1}||\leq |a_{1}-b_{1}|$.
	The  comparison matrix of $A_{1}$ is defined as $\langle \bar{a}_{ij}\rangle=|\bar{a}_{ij}| $ if $i=j$ and $\langle \bar{a}_{ij}\rangle=-|\bar{a}_{ij}| $ if $i \neq j$; a  $Z$-matrix if all of its non-diagonal elements are less than equal to zero; an $M$-matrix if $A_ 1^{-1}\geq 0$ as well as   $Z$-matrix; an $H$-matrix, if   $\langle A_ 1 \rangle$ is an $M$-matrix and an $H_+$-matrix if  $ A_{1}$  is   an $H$-matrix as well as $\bar   {a}_{ii} ~\textgreater~ 0 ~\forall ~i \in \{1,2,\ldots,n\}$; a $P$-matrix if all its principle minors are positive such that $det({A_1}_{\alpha_1 \alpha_1}) ~\textgreater~ 0$ $\forall$ $\alpha_1 \subseteq \{1,2,\ldots, n\}$.
	The splitting $A_{1} = M_{1}-N_{1} $ is called an $M$-splitting if $M_{1}$ is a nonsingular $M$-matrix and $N_{1} \geq 0$; an  $H$-splitting if $\langle M_{1} \rangle -|N_1| $ is an $M$-matrix; an $H$-compatible splitting if  $\langle A_{1} \rangle =  \langle M_1 \rangle - | N_1 | $; splitting is a $H$-splitting if it is a $H$-compatible of an  $H$-matrix, while the converse is not true. 
	
	\begin{lem}\cite{rashid2022}\label{lem0}
		Let $a_{1}, b_1 \in R^{n} $.  $a_{1}\geq 0$, $b_{1}\geq 0$, $a_1^{T}b_{1}=0$ if and only if  $a_{1}+b_{1}=|a_{1}-b_{1}|$.
	\end{lem}
	\begin{lem}\label{lem1}\cite{Frommer1992}
		Suppose  $A_{1}, B_{1} \in {R}^{n\times n}$. If $A_{1}$ and $B_{1}$ are $M$ and $Z$-matrices, respectively, with $A_{1} \leq B_{1}$ then $B_{1}$ is an $M$-matrix.
		If $A_{1}$ is an $H$-matrix then $|A_{1}^{-1}|\leq \langle A_{1}\rangle^{-1}$.
		If $A_{1} \leq B_{1}$, then $\rho(A_{1}) \leq \rho(B_{1})$.
	\end{lem}
	\begin{lem}\label{lem2}\cite{Xi2021}
		Let $A_{1}\in {R}^{n\times n}$ be  an $M$-matrix  and $A_{1}=M_{1}-N_{1}$ be an $M$-splitting. Let $\rho$ be the spectral radius, then ~$\rho(M_{1}^{-1}N_{1})$ $\textless $ $1$.
	\end{lem}	
	\begin{lem}\label{lem4} \cite{Frommer1992}
		Suppose  $A_{1} \geq 0 $. If there exist  $v ~\textgreater~ 0 \in {R}^{n}$ and a scalar $\alpha_{1}~\textgreater ~ 0$  such that $A_{1}v \leq  \alpha_{1} v$, then $\rho(A_{1}) \leq \alpha_{1} $. Moreover, if $ A_{1}v ~\textless ~v$, then $\rho(A_{1})~\textless~ 1$.
	\end{lem}
	\section{Main results}			
	For a given vector $s\in R^{n }$, we indicate the vectors $s_{+}=max\{0,s\}$ and $A_{1}=(M_1 +I- L_1)-(N_1+I- L_1)$, where $I$ is the identity matrix of order $n$ and $ L_1$ is the strictly lower triangular matrix of $A_1$. In the following result, we convert the  LCP$(q, A_{1})$ into a fixed point formulation. 
	\begin{thm}\label{thm1} Let $A_{1}\in R^{n\times n} $ with the splitting  $A_{1}= (M_1+I- L_1)- (N_1+I- L_1)$. Let  $z=
		\tau(|s|+s)$, $\omega =\Omega_1(|s|-s)$ and the matrix $(M_1+\Omega_{1} +I- L_1)$ be a nonsingular, then the  equivalent formulation of  the LCP$(q,A_{1})$ in form of fixed point equation is \begin{eqnarray}\label{eq2}
			s=(M_1+\Omega_{1}+I- L_1)^{-1}[(N_{1}+I- L_1)s+ (\Omega_1-A_{1})|s|- rq].  
		\end{eqnarray}
	\end{thm}
	\begin{proof}
		We have $z=\tau(|s|+s)$ and  $\omega =\Omega_1(|s|-s)$, from Equation $(\ref{eq1})$ we obtain 
		\begin{align*}
			\Omega_1(|s|-s) &=A_{1}\tau(|s|+s) + q\\
			(A_{1}\tau+\Omega_1)s	&=(\Omega_{1}-A_{1}\tau)|s|-q\\
			((M_{1}+I- L_1)\tau+\Omega_1)s&=(N_1+I- L_1)\tau s+(\Omega_1-A_{1}\tau)|s|-q.
		\end{align*}
		Let $\tau=\frac{1}{r}$, the above equation can be rewritten as,
		$$s=(M_{1}+I- L_1+\Omega_1)^{-1}[(N_1+I- L_1) s+(\Omega_1-A_{1})|s|-rq]. $$
	\end{proof}
	\noindent In the following, Based on Equation ($\ref{eq2}$), we propose an iteration method which is known as Method 3.1 to solve the LCP$(q, A_{1})$.
	\begin{mthd}\label{mthd1}
		Let $A_{1}=(M_1 +I- L_1)-(N_1+I- L_1)$ be a splitting of the matrix $A_{1} \in R^{n\times n} $.
		 Suppose that $(M_1+\Omega_{1}+I- L_1)$ is a nonsingular matrix. Then we use
		the following equation for Method 3.1 is
			\begin{equation}\label{eq003}
			s^{(k+1)}=(M_1+\Omega_{1}+I- L_1)^{-1}[(N_{1}+I- L_1)s^{(k)}+ (\Omega_1-A_{1})|s^{(k)}|- rq]
		\end{equation}
	Let Residual be the Euclidean norm of the error vector, which is defined as follows:  $$ Res(z^{(k)})=\|min(z^{(k)}, A_{1}z^{(k)}+q) \|_{2}.$$ Consider a nonnegative initial vector $z^{(0)}\in R^n$.   The iteration process continues until the iteration sequence $\{z^{(k)}\}_{k=0}^{+\infty} \subset R^n$ converges.   For  $k=0,1,2,\ldots$, the iterative
	process continues until the iterative sequence  $z^{(k+1)}\in R^{n}$ converges. The iteration process stops if $Res(z^{(k)})$ $\textless $ $ \epsilon $. For computing $z^{(k+1)}$ we use the following  steps.
	 
		\noindent \textbf{Step 1:} Given an initial vector $s^{(0)} \in R^{n}$, $\epsilon ~\textgreater ~ 0 $  and set $ k=0 $.
			
			\noindent \textbf{Step 2:} Using the following scheme, create the sequence $z^{(k)}$:
			\begin{equation}\label{eq3}
				s^{(k+1)}=(M_1+\Omega_{1}+I- L_1)^{-1}[(N_{1}+I- L_1)s^{(k)}+ (\Omega_1-A_{1})|s^{(k)}|- rq]
			\end{equation}
			and set $z^{(k+1)}=\frac{1}{r}(|s^{(k+1)}|+s^{(k+1)})$, where $z^{(k)}$ is a $k^{th}$ approximate solution of LCP$(q, A_{1})$ and $s^{(k}$ is a $k^{th}$ approximate solution of Equation $(\ref{eq2})$.
			
				\noindent \textbf{Step 3:} Stop if $ Res(z^{(k)})$ $ \textless $ $\epsilon$; otherwise, set $k=k+1$ and return to step 2.
	\end{mthd}

	\noindent Furthermore, the proposed Method \ref{mthd1} offers a generic framework for solving LCP$(q, A_1)$. We created a new family of accelerated modulus-based relaxation methods using  matrix splitting. In particular, we express the system matrix $A_{1}$ as $A_{1}=(M_1+I- L_1)-(N_{1}+I- L_1)$. Then
	\begin{enumerate}
		\item when $M_{1}=A_{1}$, $N_{1}=0$,  $\Omega_{1}=I$ and $r=1 $, Equation ($\ref{eq003}$) gives the  new accelerated  modulus iteration method is
		\begin{eqnarray*}
			s^{(k+1)}=(A_1+2I- L_1)^{-1}[(I- L_1)s^{(k)}+ (I-A_{1})|s^{(k)}|-q].  
		\end{eqnarray*}
		\item when $M_{1}=A_{1}$, $N_{1}=0$,  $\Omega_{1}=\alpha_1I$ and $r=1 $, Equation ($\ref{eq003}$) gives the  new accelerated  modified modulus-based  iteration method is
		\begin{eqnarray*}
			s^{(k+1)}=(A_1+(\alpha_1+1)I- L_1)^{-1}[(I- L_1)s^{(k)}+ (\alpha_1I-A_{1})|s^{(k)}|-q]. 
		\end{eqnarray*}
		\item when $M_{1}=D_1$, $N_{1}=L_1+U_1$ and $r=2$, Equation ($\ref{eq003}$) gives the  new accelerated modulus-based Jacobi iteration method is
		\begin{eqnarray*}
			s^{(k+1)}=(D_1+\Omega_{1}+I- L_1)^{-1}[(U_{1}+I)s^{(k)}+ (\Omega_1-A_{1})|s^{(k)}|- 2q].
		\end{eqnarray*}
		\item when $M_{1}=D_1-L_1$, $N_{1}=U_1$ and $r=2$, Equation ($\ref{eq003}$) gives the new accelerated  modulus-based  Gauss-Seidel iteration (NAMGS) method  is
		\begin{eqnarray*}
			s^{(k+1)}=(D_1-2L_1+\Omega_{1}+I)^{-1}[(U_{1}+I- L_1)s^{(k)}+ (\Omega_1-A_{1})|s^{(k)}|- 2q].
		\end{eqnarray*}
		\item when $M_{1}=(\frac{1}{\alpha_1}D_1-L_1)$ and $N_{1}=(\frac{1}{\alpha_1}-1)D_1+U_{1}$, Equation ($\ref{eq3}$) gives the  new accelerated  modulus-based  successive   over-relaxation  iteration (NAMSOR) method  is
		\begin{eqnarray*}
			\begin{split}
				s^{(k+1)}&=(D_1-2\alpha_1 L_1+\alpha_1 \Omega_{1}+\alpha_1I))^{-1}[((1-\alpha_1)D_1+\alpha_{1}U_{1}\\&+\alpha_1I- L_1)s^{(k)}+ (\alpha_{1}\Omega_{1}-\alpha_1A_1)s^{(k)}- 2\alpha_1 q]. 
			\end{split}
		\end{eqnarray*}
		\item when $M_{1}=(\frac{1}{\alpha_1})(D_1-\beta_1  L_1)$ and $N_{1}=(\frac{1}{\alpha_1})[(1-\alpha_1)D_1+(\alpha_1 - \beta_1 ) L_1+\alpha_1 U_{1}]$, Equation ($\ref{eq3}$) gives the new accelerated  modulus-based  accelerated b overrelaxation  iteration (NAMAOR) method  is
		\begin{eqnarray*}
			\begin{split}
				s^{(k+1)}&=(D_1-(\beta_1 +\alpha_1)L_1+\alpha_1 \Omega_{1}+\alpha_1I)^{-1}[((1-\alpha_1)D_1+(2\alpha_{1}-\beta_1)L_{1}\\&+\alpha_{1}U_{1}+\alpha_1I)s^{(k)}+ (\alpha_{1}\Omega_{1}-\alpha_1A_1 )s^{(k)}- 2\alpha_1 q].  
			\end{split}
		\end{eqnarray*}
	\end{enumerate}
	The NAMAOR method clearly converts into the new accelerated modulus-based successive over-relaxation (NAMSOR) method, Gauss-Seidel (NAMGS) method, and Jacobi method when $(\alpha_1, \beta_1)$ takes the values $(\alpha_1, \alpha_1)$, $(1, 1)$, and $(1, 0)$, respectively.
	\section {Convergence analysis}
	\noindent	In the following result, we prove the convergence conditions when the system matrix $A_{1}$ is a $P$-matrix.
	\begin{thm}\label{thm2}
		Let $A_{1} \in R^{n \times n}$ be a $P$-matrix and  $s^{*}$ be the solution  of  Equation $(\ref{eq2})$. Let $	\rho(|(M+{I- L_1+\Omega_{1}})^{-1}|(|N+I- L_1|+|\Omega_{1}-A_{1}|))~\textless ~1 $. Then the sequence $\{s^{(k)}\}^{+\infty}_{k=1}$  generated by Method  $\ref{mthd1}$ converges to the solution $s^{*}$ for any initial vector $s^{(0)}\in R^{n}$.	
	\end{thm}
	\begin{proof}
		Let $s^*$ be the solution of Equation $(\ref{eq2})$, then error is 
		\begin{equation*}
			\begin{split}
				s^{(k+1)}-s^*&=(M_{1}+{I- L_1+\Omega_{1}})^{-1}[(N_{1}+I- L_1)(s^{(k)}-s^*)\\&+ (\Omega_{1}-A_{1})(|s^{(k)}|-|s^{*}|)]\\
				|s^{(k+1)}-s^*|&=|(M_{1}+{I- L_1+\Omega_{1}})^{-1}[(N_{1}+I- L_1)(s^{(k)}-s^*)+ (\Omega_{1}\\&-A_{1})(|s^{(k)}|-|s^{*}|)]|\\
				&\leq |(M_{1}+{I- L_1+\Omega_{1}})^{-1}|(|(N_{1}+I- L_1)(s^{(k)}-s^*)|+ |\Omega_{1}\\&-A_{1})(|s^{(k)}-s^*|)|\\ 
				&\leq |(M_{1}+{I- L_1+\Omega_{1}})^{-1}|(|N_{1}+I- L_1|\\&+ |\Omega_{1}-A_{1}|)|s^{(k)}-s^*|\\ 
				|s^{(k+1)}-s^*|	&\textless |s^{(k)}-s^*|.
			\end{split}	
		\end{equation*}
		Therefore, the sequence $\{s^{(k)}\}^{+\infty}_{k=1}$ converges to the solution $s^*$.
	\end{proof}
	\noindent When the system matrix $A _1$ is an ~$H _+$-matrix, the following result discusses the convergence domain of $\Omega_1$ for a new accelerated modulus-based iteration method.
	\begin{thm}{\label{thm3}}
		Let~ $A_1$ be an $H_{+ }$-matrix and   $A_{1}=M_{1}-N_{1}=(M_{1}+I- L_1)-(N_{1}+I- L_1)$ be an $H$-compatible of the matrix $A_{1}$, such that $\langle A_{1} \rangle= \langle M_{1} + I- L_1 \rangle -|N+I- L_1|$ and either one of the following conditions holds:\\
		{(1)} $\Omega_{1} \geq D_{1}$;\\
		(2) $\Omega_{1} \textless  D_{1}$ and $2\Omega_{1}-D_{1}-|B|$,  is an $M$- matrix, $B=L_{1}+U_{1}$.
		Then the sequence $\{s^{(k)}\}^{+\infty}_{k=1}$  generated by Method \ref{mthd1} converges to the solution $s^{*}$ for any initial vector $s^{(0)}\in R^{n}$.
	\end{thm}
	\begin{proof}
		Let $A_{1}=M_{1}-N_{1}=(M_{1}+I- L_1)-(N_{1}+I- L_1)$ and it holds that 
		\\ $\langle A_{1} \rangle\leq  \langle M_{1} + I- L_1 \rangle \leq diag(M_{1}+I- L_1)$, 
		$(M_{1}+I- L_1)$ is an $H_{+}$-matrix. and it holds that 
		$$ |(\Omega_{1}+M_{1}+I- L_1)^{-1}| \leq (\Omega_{1}+\langle M_{1}\rangle+I- L_1)^{-1}.$$
		From Theorem $\ref{thm2}$,	let $T=|(M_{1}+{I- L_1+\Omega_{1}})^{-1}|(|N_{1}+I- L_1|+ |\Omega_{1}-A_{1}|)$, then
		\noindent 
		\begin{equation*}
			\begin{split}
				T & =|(  M_{1} +\Omega_{1}+I- L_1)^{-1}|[|N_1+I- L_1|+|\Omega_{1}-A_{1}|]\\
				&\leq (\langle M_{1}\rangle+\Omega_{1}+I- L_1)^{-1}[|N_1+I- L_1|+|\Omega_{1}-A_{1}|]\\
				&\leq (\langle M_{1}\rangle+\Omega_{1}+I- L_1)^{-1}[|N_1+I- L_1|+|\Omega_{1}-D_1+ L_1+ U_1|]\\
				&\leq (\langle M_{1}\rangle+\Omega_{1}+I- L_1)^{-1}[(\langle M_{1}\rangle+\Omega_{1}+I- L_1)-(\langle M_{1}\rangle+\Omega_{1}+I- L_1)\\&+|N_1+I- L_1|+|\Omega_{1}-D_1|+| L_1+ U_1|].\\
			\end{split}
		\end{equation*}	
		Case 1. $\Omega_{1} \geq D_{1} $,
		\begin{equation*}
			\begin{split}	
				&\leq I-(\langle M_{1}\rangle+\Omega_{1}+I- L_1)^{-1}[(\langle M_{1}\rangle+I- L_1)-|N_1+I- L_1|+D_1-| L_1+ U_1|]\\
				&\leq I-2(\Omega_{1}+\langle M_{1}\rangle+I- L_1)^{-1}\langle A_{1}\rangle.\\		
			\end{split}
		\end{equation*}
		Since $\langle A_{1}\rangle$ is an $M$-matrix, then there exists a positive vector $v~ \textgreater ~0$ such that $$\langle A_{1}\rangle v ~\textgreater ~0.$$
		Therefore $$Tv\leq (I-2(\Omega_{1}+\langle M_{1}\rangle+I- L_1)^{-1}\langle A_{1}\rangle)v~ \textless~  v.$$
		By using the Lemma $\ref{lem4}$, we are able to determine that $\rho(T)\textless  1$. \\
		Case 2. $\Omega_{1}~ \textless~  D_{1}$ and $\langle A_{1} \rangle +2\Omega_{1}-D_{1}-|B|$ is an $M$-matrix. Then,
		\begin{equation*}
			\begin{split}
				T 
				&\leq (\langle M_{1}\rangle+\Omega_{1}+I- L_1)^{-1}[(\langle M_{1}\rangle+\Omega_{1}+I- L_1)-(\langle M_{1}\rangle+\Omega_{1}+I- L_1)\\&+|N_1+I- L_1|+|\Omega_{1}-D_1|+| L_1+ U_1|]\\
				&\leq I-(\langle M_{1}\rangle+\Omega_{1}+I- L_1)^{-1}[(\langle M_{1}\rangle+I- L_1)-|N_1+I- L_1|\\&+2\Omega_{1}-D_1-| L_1+ U_1|].\\
			\end{split}
		\end{equation*}	
		Since $[\langle A_1 \rangle+2\Omega_{1}-|D_1|-| L_1+ U_1|]$ is an $M$-matrix. Then there exists a positive vector $v~ \textgreater ~0$ such that $$[\langle A_1 \rangle+2\Omega_{1}-|D_1|-| L_1+ U_1|] v ~\textgreater ~0.$$
		Therefore $$Tv\leq I-(\Omega_{1}+\langle M_{1}\rangle+I- L_1)^{-1}[\langle A_1 \rangle+2\Omega_{1}-D_1-| L_1+ U_1|] v~  \textless ~ v.$$
		By using the Lemma $\ref{lem4}$, we are able to determine that $\rho(T)\textless  1$. \\
		Because of this, according to Theorem $\ref{thm2}$, the iteration sequence $\{s^{(k)}\}^{+\infty}_{k=1}$ generated by Method $\ref{mthd1}$ converges to $s^{*}$ for  any initial vector $s^{(0)}.$
	\end{proof}
	\section{Numerical examples}
	IT denotes the number of iteration steps, while CPU is the CPU time in seconds. This section includes two numerical examples to show the efficiency of our proposed method.
	We consider the LCP$(q, A_{1})$, which always has a unique solution. Let $A_{1}= P_{1}+\delta_{1}I$ and $q=-A_{1}z^*$, where $z^*=(1,2,\ldots,1,2,\ldots)\in R^n$ is the unique solution of Equation ($\ref{eq1}$). Let $s^{(0)}=(1,0,\ldots 1,0,\ldots)^T\in R^{n}$ be initial vector and set $\epsilon=10^{-5}$.
	The proposed methods (NAMGS and NAMSOR) are compared to the modulus-based   Gauss-Seidel (MGS) method  and the successive over-relaxation (MSOR) method   \cite{Bai2010}, which are effective in solving LCP$(q, A_{1})$.\\ 
	\noindent Matlab version 2021a on an Acer Desktop (Intel(R) Core(TM) i7-8700 CPU @ 3.2 GHz 3.19GHz, 16.00GB RAM) is used  for all calculations.
	The numerical results for the new accelerated modulus-based iteration method and modulus-based matrix splitting method in \cite{Bai2010} are listed in Tables 1 and 2.
	\begin{ex}\label{ex1}
		The system matrix $A_{1} \in {R}^{n\times n}$  is generated by $A_{1}= P_{1}+\delta_{1}I$, 
		where $\delta_{1}$  is nonnegative real parameter and \\
		$P_{1}=
		\begin{bmatrix}
			L_{1} & -I_{2} &0 & \ldots &0  \\
			-I_{2} & L_{1} & -I_{2} &\ldots &0  \\  
			0& -I_{2} & L_{1} &-I_{2}   &0\\ 
			0& \ldots & -I_{2} & \ddots & -I_{2}\\ 
			0& \ldots & 0 &-I_{2} &L_{1}\\ 
		\end{bmatrix} $,
		$L_{1}=
		\begin{bmatrix}
			4 &-1  & \ldots&\ldots &0   \\
			-1&  4 & -1 & \ldots& 0  \\  
			0&  -1& 4 &-1 &  0\\  
			0& \ldots &-1  & \ddots & -1\\ 
			0& \ldots & \ldots &-1 &4\\ 
		\end{bmatrix} $ ,\\
		\noindent where $P_{1}\in {R}^{n\times n}$, $L_{1}\in {R}^{m\times m}$ 
		and $I_{2}$ is the identity matrix of order $m$.  
	\end{ex}
	\begin{table}[h!]
		\caption{Results for  MGS and MSOR methods \cite{Bai2010} and  NAMGS and NAMSOR methods,  $\delta_{1}=4$.}
		\centering
		\begin{tabular}{@{}|l|l| l l l l l l|@{}}
			\hline
			&\textbf{n} &${100}$   & ${900}$   & ${2500}$ & ${3600}$& ${6400}$ & ${10000}$ \\ [.9ex] 
			\hline\hline
			$\textbf{MGS}$&$\textbf{IT} $&36 & 40 & 41 & 41& 42 & 42\\
			$\alpha=1$	& $\textbf{CPU}$ & 0.0039 & 0.0293  & 0.2699 &0.6400&2.0785&4.5712 \\
			&$\textbf{Res} $&9.7e-06 &8.0e-06 &7.9e-06 &8.9e-06&7.4e-06&8.4e-06\\
			$\textbf{NAMGS}$&$\textbf{IT}$&16  &  17&17  &17&18&18\\
			$\alpha_1=1$	&$\textbf{CPU} $ & 0.0024 & 0.0146 &0.1089 &0.2688&0.8179&2.0175\\
			&$\textbf{Res} $&6.3e-06 &6.3e-06 &8.5e-06&9.5e-06& 4.5e-06&5.1e-06\\
			\hline 
			
			$\textbf{MSOR}$&$\textbf{IT} $&15 & 17 & 18 & 18& 18 &19\\
			$\alpha=0.85$	& $\textbf{CPU}$ & 0.0034 & 0.0141  & 0.1201 &0.2903&0.7826&2.0791 \\
			&$\textbf{Res} $&9.5e-06 &7.6e-06 &5.3e-06 &6.5e-06&8.9e-06&4.3e-06\\
			$\textbf{NAMSOR}$&$\textbf{IT}$&12  &  12&13  & 13 & 13 & 13\\
			$\alpha_1=0.91$	&$\textbf{CPU} $ & 0.0028 & 0.0116 &0.0905 &0.2017&0.6467&2.1918\\
			&$\textbf{Res} $&2.7e-06 &7.6e-06 &3.2e-06&3.7e-06& 4.8e-06&5.9e-06\\
			\hline 
		\end{tabular}
	\end{table}
	\begin{ex} The system matrix $A_{1}\in {R}^{n\times n}$  is generated by $A_{1}= P_{1}+\delta_{1}I$, 
		where $\delta_{1}$  are nonnegative real parameters and \\ 
		$P_{1}=
		\begin{bmatrix}
			L_{1} & -0.5I_{2} &0 & \ldots  \\
			-1.5I_{2} & L_{1} & -0.5I_{2} &\ldots  \\  
			\vdots& -1.5I_{2} & \ddots &-0.5I_{2}   \\ 
			0& \ldots & -1.5I_{2} & L_{1} \\  
		\end{bmatrix} $,
		$L_{1}=
		\begin{bmatrix}
			4 &-0.5  & \ldots&\ldots  \\
			-1.5&  4 & -0.5 & \ldots \\  
			\vdots&  -1.5& \ddots &-0.5 \\  
			0& \ldots &-1.5  & 4\\ 
		\end{bmatrix} $ \\ \ \\ $P_{1}\in {R}^{n\times n}$, $L_{1}\in {R}^{m\times m}$
		and
		$I_{2}$ is the identity matrix of order $m$.
	\end{ex}
	\begin{table}[h!]
		
		\caption{Results for  MGS and MSOR methods \cite{Bai2010} and  NAMGS and NAMSOR  methods,    $\delta_{1}=4$.}
		\centering
		\begin{tabular}{@{}|l|l| l l l l l l|@{}}
			\hline
			&\textbf{n} &${100}$   & ${900}$   & ${2500}$ & ${3600}$& ${6400}$ & ${10000}$ \\ [.9ex] 
			\hline\hline
			$\textbf{MGS}$&$\textbf{IT} $&24 & 26 & 26& 26& 27 &27\\
			$\alpha=1$	& $\textbf{CPU}$ & 0.0034 & 0.0208  & 0.1707 &0.3961&1.1542& 2.9471\\
			&$\textbf{Res} $&9.0e-06 &6.5e-06 &8.7e-06 &9.6e-06&6.5e-06&7.3e-06\\
			$\textbf{NAMGS}$&$\textbf{IT}$&12  &  12&13  &13&13&13\\
			$\alpha_1=1$	&$\textbf{CPU} $ & 0.0026 & 0.0141 &0.0837 &0.1998&0.2987&1.4461\\
			&$\textbf{Res} $&3.8e-06 &8.4e-06 &3.1e-06&3.5e-06& 4.4e-06&5.3e-06\\
			\hline 
			$\textbf{MSOR}$&$\textbf{IT} $&14 & 14 & 15 & 15& 15 &15\\
			$\alpha=0.88$	& $\textbf{CPU}$ & 0.0029 & 0.0123  & 0.1040 &0.2505&0.6543&1.6363 \\
			&$\textbf{Res} $&3.8e-06 &8.2e-06 &4.1e-06 &4.6e-06&5.5e-06&6.3e-06\\
			$\textbf{NAMSOR}$&$\textbf{IT}$&8  &  9&9 &9&9&9\\
			$\alpha_1=0.88$	&$\textbf{CPU} $ & 0.0022 & 0.0096 &0.0655 &0.1409&0.0441&1.7340\\
			&$\textbf{Res} $&2.2e-06 &1.5e-06 &2.7e-06&3.7e-06& 4.6e-06&5.9e-06\\
			\hline 
		\end{tabular}
	\end{table}
	\noindent From Tables 1 and 2, we can observe that the iteration steps required by our proposed NAMGS and NAMSOR methods have  lesser number of iteration steps, faster processing (CPU time) and a greater computational efficiency than  the MGS and MSOR methods in \cite{Bai2010} respectively.
	\section{Conclusion} 
	In this article, we present a class of new accelerated modulus-based iteration methods for solving the LCP$(q, A _1)$ based on matrix splitting. The large and sparse structure of $A_ 1$ is preserved throughout the iteration process by this iteration form. Additionally,  when system matrix $A_{1}$ is an $H_{+}$-matrix, we demonstrate some convergence conditions. Finally, two numerical examples are provided to illustrate the effectiveness of the proposed methods.\\ \  \\
	\noindent \textbf{Conflict of interest} There are no conflicts of interest declared by the authors.\\  \ \\
	\noindent \textbf{Acknowledgment.}  The first author wishes to thank the University Grants Commission (UGC), Government of India, under the SRF fellowship program No. 1068/(CSIR-UGC NET DEC. 2017).
		\bibliographystyle{plain} 
	\bibliography{bharatt}

\begin{thebibliography}{10}

\bibitem{Ahn1981}
BH~Ahn.
\newblock Solution of nonsymmetric linear complementarity problems by iterative
  methods.
\newblock {\em Journal of optimization Theory and Applications}, 33:175--185,
  1981.

\bibitem{rashid2022}
R~Ali, I~Khan, A~Ali, and A~Mohamed.
\newblock Two new generalized iteration methods for solving absolute value
  equations using m-matrix.
\newblock {\em AIMS mathematics}, 7(5):8176--8187, 2022.

\bibitem{Bai2010}
ZZ~Bai.
\newblock Modulus-based matrix splitting iteration methods for linear
  complementarity problems.
\newblock {\em Numerical Linear Algebra with Applications}, 17(6):917--933,
  2010.

\bibitem{Bai1997}
ZZ~Bai and DJ~Evans.
\newblock Matrix multisplitting relaxation methods for linear complementarity
  problems.
\newblock {\em International Journal of Computer Mathematics},
  63(3-4):309--326, 1997.

\bibitem{Berman1979}
A~Berman and RJ~Plemmons.
\newblock {\em Nonnegative matrices in the mathematical sciences}.
\newblock SIAM, 1994.

\bibitem{Das2014}
AK~Das.
\newblock Properties of some matrix classes based on principal pivot transform.
\newblock {\em Annals of Operations Research}, 243:375--382, 2016.

\bibitem{das2018some}
AK~Das, Deepmala, and R~Jana.
\newblock Some aspects on solving transportation problem.
\newblock {\em Yugoslav Journal of Operations Research}, 30(1):45--57, 2020.

\bibitem{Das2017}
AK~Das and R~Jana.
\newblock Finiteness of criss-cross method in complementarity problem.
\newblock In {\em Mathematics and Computing: Third International Conference,
  ICMC 2017, Haldia, India, January 17-21, 2017, Proceedings 3}, pages
  170--180. Springer, 2017.

\bibitem{Das2016}
AK~Das and R~Jana.
\newblock On generalized positive subdefinite matrices and interior point
  algorithm.
\newblock In {\em Operations Research and Optimization: FOTA 2016, Kolkata,
  India, November 24-26 1}, pages 3--16. Springer, 2018.

\bibitem{Dong2009}
JL~Dong and MQ~Jiang.
\newblock A modified modulus method for symmetric positive-definite linear
  complementarity problems.
\newblock {\em Numerical Linear Algebra with Applications}, 16(2):129--143,
  2009.

\bibitem{dutta2022column}
A~Dutta, R~Jana, and AK~Das.
\newblock On column competent matrices and linear complementarity problem.
\newblock In {\em Proceedings of the Seventh International Conference on
  Mathematics and Computing: ICMC 2021}, pages 615--625. Springer, 2022.

\bibitem{Xi2021}
XM~Fang.
\newblock General fixed-point method for solving the linear complementarity
  problem.
\newblock {\em AIMS Mathematics}, 6(11):11904--11920, 2021.

\bibitem{Frommer1992}
A~Frommer and DB~Szyld.
\newblock H-splittings and two-stage iterative methods.
\newblock {\em Numerische Mathematik}, 63:345--356, 1992.

\bibitem{Hadjidimos2009}
A~Hadjidimos and M~Tzoumas.
\newblock Nonstationary extrapolated modulus algorithms for the solution of the
  linear complementarity problem.
\newblock {\em Linear algebra and its applications}, 431(1-2):197--210, 2009.

\bibitem{Hadjidimos2018}
A~Hadjidimos and LL~Zhang.
\newblock Comparison of three classes of algorithms for the solution of the
  linear complementarity problem with an h+-matrix.
\newblock {\em Journal of Computational and Applied Mathematics}, 336:175--191,
  2018.

\bibitem{Hong2016}
JT~Hong and CL~Li.
\newblock Modulus-based matrix splitting iteration methods for a class of
  implicit complementarity problems.
\newblock {\em Numerical Linear Algebra with Applications}, 23(4):629--641,
  2016.

\bibitem{Huang2016}
N~Huang and C~Ma.
\newblock The modulus-based matrix splitting algorithms for a class of weakly
  nonlinear complementarity problems.
\newblock {\em Numerical Linear Algebra with Applications}, 23(3):558--569,
  2016.

\bibitem{Jana2019}
R~Jana, AK~Das, and A~Dutta.
\newblock On hidden z-matrix and interior point algorithm.
\newblock {\em Opsearch}, 56:1108--1116, 2019.

\bibitem{jana2018processability}
R~Jana, AK~Das, and S~Sinha.
\newblock On processability of lemke’s algorithm.
\newblock {\em Applications and Applied Mathematics: An International Journal
  (AAM)}, 13(2):31, 2018.

\bibitem{jana2021more}
R~Jana, A~Dutta, and AK~Das.
\newblock More on hidden z-matrices and linear complementarity problem.
\newblock {\em Linear and Multilinear Algebra}, 69(6):1151--1160, 2021.

\bibitem{Kappel1986}
NW~Kappel and LT~Watson.
\newblock Iterative algorithms for the linear complementarity problem.
\newblock {\em International journal of computer mathematics},
  19(3-4):273--297, 1986.

\bibitem{Kumar2022}
B~Kumar, Deepmala, and AK~Das.
\newblock On general fixed point method based on matrix splitting for solving
  linear complementarity problem.
\newblock {\em Journal of Numerical Analysis and Approximation Theory},
  51(2):189--200, 2022.

\bibitem{das2023more}
B~Kumar, Deepmala, and AK~Das.
\newblock More on modulus based iterative method for solving implicit
  complementarity problem.
\newblock {\em arXiv preprint arXiv:2303.12519}, 2023.

\bibitem{Kumar2023i}
B~Kumar, Deepmala, and AK~Das.
\newblock Projected fixed point iterative method for large and sparse
  horizontal linear complementarity problem.
\newblock {\em Indian Journal of Pure and Applied Mathematics}, pages 1--10,
  2023.

\bibitem{kumar2022error}
B~Kumar, Deepmala, A~Dutta, and AK~Das.
\newblock Error bound for the linear complementarity problem using plus
  function.
\newblock {\em arXiv preprint arXiv:2209.00377}, 2022.

\bibitem{Kumar2023}
B~Kumar, A~Dutta, and AK~Das.
\newblock More on matrix splitting modulus-based iterative methods for solving
  linear complementarity problem.
\newblock {\em OPSEARCH}, pages 1--18, 2023.

\bibitem{Li2013}
W~Li.
\newblock A general modulus-based matrix splitting method for linear
  complementarity problems of h-matrices.
\newblock {\em Applied Mathematics Letters}, 26(12):1159--1164, 2013.

\bibitem{Ma2016}
C~Ma and N~Huang.
\newblock Modified modulus-based matrix splitting algorithms for a class of
  weakly nondifferentiable nonlinear complementarity problems.
\newblock {\em Applied Numerical Mathematics}, 108:116--124, 2016.

\bibitem{Mangasarian1977}
OL~Mangasarian.
\newblock Solution of symmetric linear complementarity problems by iterative
  methods.
\newblock {\em Journal of Optimization Theory and Applications},
  22(4):465--485, 1977.

\bibitem{Mezzadri2020}
F~Mezzadri and E~Galligani.
\newblock Modulus-based matrix splitting methods for horizontal linear
  complementarity problems.
\newblock {\em Numerical Algorithms}, 83(1):201--219, 2020.

\bibitem{Najafi2013}
HS~Najafi and SA~Edalatpanah.
\newblock Modification of iterative methods for solving linear complementarity
  problems.
\newblock {\em Engineering computations}, 30(7):910--923, 2013.

\bibitem{neogy2011singular}
SK~Neogy and AK~Das.
\newblock On singular n0-matrices and the class q.
\newblock {\em Linear algebra and its applications}, 434(3):813--819, 2011.

\bibitem{Neogy2012}
SK~Neogy, AK~Das, and A~Gupta.
\newblock Generalized principal pivot transforms, complementarity theory and
  their applications in stochastic games.
\newblock {\em Optimization Letters}, 6:339--356, 2012.

\bibitem{Van1980}
WMG Van~Bokhoven.
\newblock A class of linear complementarity problems is solvable in polynomial
  time.
\newblock {\em unpublished paper, Dept. of electrical engineering, university
  of technology, the Netherlands}, 1980.

\bibitem{Shilang2021}
S~Wu and C~Li.
\newblock A class of new modulus-based matrix splitting methods for linear
  complementarity problem.
\newblock {\em Optimization Letters}, pages 1--17, 2022.

\bibitem{Wu2018}
SL~Wu and P~Guo.
\newblock Modulus-based matrix splitting algorithms for the
  quasi-complementarity problems.
\newblock {\em Applied Numerical Mathematics}, 132:127--137, 2018.

\bibitem{Xia2015}
Z~Xia and C~Li.
\newblock Modulus-based matrix splitting iteration methods for a class of
  nonlinear complementarity problem.
\newblock {\em Applied Mathematics and Computation}, 271:34--42, 2015.

\bibitem{Xu2015}
WW~Xu.
\newblock Modified modulus-based matrix splitting iteration methods for linear
  complementarity problems.
\newblock {\em Numerical Linear Algebra with Applications}, 22(4):748--760,
  2015.

\bibitem{Yuan2003}
D~Yuan and Y~Song.
\newblock Modified aor methods for linear complementarity problem.
\newblock {\em Applied Mathematics and Computation}, 140(1):53--67, 2003.

\bibitem{Zhang2011}
LL~Zhang.
\newblock Two-step modulus-based matrix splitting iteration method for linear
  complementarity problems.
\newblock {\em Numerical Algorithms}, 57:83--99, 2011.

\bibitem{Zhang2013}
LL~Zhang and ZR~Ren.
\newblock Improved convergence theorems of modulus-based matrix splitting
  iteration methods for linear complementarity problems.
\newblock {\em Applied Mathematics Letters}, 26(6):638--642, 2013.

\bibitem{Zheng2013}
N~Zheng and JF~Yin.
\newblock Accelerated modulus-based matrix splitting iteration methods for
  linear complementarity problem.
\newblock {\em Numerical Algorithms}, 64(2):245--262, 2013.

\end{thebibliography}
\end{document}